\documentclass[11pt]{amsart}
\usepackage{amsbsy,amssymb,amscd,amsfonts,latexsym,amstext,delarray,
amsmath,graphicx} 
\input xypic

\newtheorem{thm}{Theorem}[section]

\newtheorem{cor}[thm]{Corollary}

\newtheorem{defn}[thm]{Definition}
\newtheorem{rem}[thm]{Remark}
\newtheorem{ex}[thm]{Example}

\numberwithin{equation}{section}

\def\bL{{\mathbb L}}

\def\Lbb{{\mathbb L}}
\def\Zbb{{\mathbb Z}}
\def\Qbb{{\mathbb Q}}
\def\Cbb{{\mathbb C}}
\def\Pbb{{\mathbb P}}
\def\Abb{{\mathbb A}}

\def\A{{\mathbb A}}
\def\C{{\mathbb C}}

\renewcommand{\P}{{\mathbb P}}

\def\cR{{\mathcal R}}

\def\Var{{\rm Var}}

\title{Graph hypersurfaces and a dichotomy in the Grothendieck ring}
\author{Paolo Aluffi}
\author{Matilde Marcolli}
\address{Department of Mathematics \\ 
Florida State University \\
Tallahassee, FL 32306, USA}
\email{aluffi\@@math.fsu.edu}
\address{Department of Mathematics  \\
California Institute of Technology \\ 
Pasadena, CA 91125, USA}
\email{matilde\@@caltech.edu}

\begin{document}
\maketitle

\begin{abstract}
The subring of the Grothendieck ring of varieties generated 
by the graph hypersurfaces of quantum field theory maps to  
the monoid ring of stable birational equivalence classes of varieties.
We show that the image of this map is the copy of $\Zbb$ generated by
the class of a point.
Thus, the span of the graph hypersurfaces in the Grothendieck
ring is nearly killed by setting the Lefschetz motive $\Lbb$ to zero, while it 
is known that graph hypersurfaces generate the Grothendieck ring
over a localization of $\Zbb[\Lbb]$ in which $\Lbb$ becomes invertible.
In particular, this shows that the graph hypersurfaces do {\em not\/} generate 
the Grothendieck ring prior to localization.

The same result yields some information on the mixed Hodge
structures of graph hypersurfaces, in the form of a constraint on the 
terms in their Deligne-Hodge polynomials.
\end{abstract}

\section{Introduction}

The interplay between perturbative quantum field theory and the theory
of motives of algebraic varieties has been extensively studied in recent
years, in particular in terms of the algebro-geometric and motivic properties of the
{\em graph hypersurfaces} associated to Feynman graphs of scalar
quantum field theories.

In particular, one of the central results in the field is the main theorem of
\cite{BeBro}, which shows that graph hypersurfaces can be arbitrarily 
complicated from the motivic viewpoint: their affine complements are, 
{\em `from the standpoint of their zeta functions, the most general schemes 
possible'\/} (\cite{BeBro}, p.~149). In rough terms, this is proven by showing 
that graph hypersurfaces generate the Grothendieck ring $K(\Var)$ of 
varieties. More precisely
(Theorem~0.6 in \cite{BeBro}), graph hypersurfaces generate $S^{-1}
K(\Var)$ as a module over the ring $S^{-1}\Zbb[\Lbb]$,
where $S$ is the (saturated) multiplicative subset of $\Zbb[\Lbb]$ generated
by $\Lbb^n-\Lbb$ for $n>1$, with $\Lbb=[\Abb^1]$ the Lefschetz-Tate motive.

The main result of this note will imply that this localization is in fact necessary
to the result of \cite{BeBro}. We will show that graph hypersurfaces
do {\em not\/} generate the Grothendieck ring as a module over $\Zbb[\Lbb]$,
and in fact they do not even generate the localization ${S'}^{-1}K(\Var)$ as a 
module over ${S'}^{-1}\Zbb[\Lbb]$,
where ${S'}$ is generated by $\Lbb^n-1$ for $n>0$. For example, we will
show that the class of an elliptic curve is not in the span of the graph
hypersurfaces if coefficients are taken in ${S'}^{-1}\Zbb[\Lbb]$.
As $S^{-1}\Zbb[\Lbb]$
is the localization of this latter ring at~$\Lbb$, it appears that localization 
at~$\Lbb$ is crucial the result of \cite{BeBro}. In view of this observation, we
propose the following\medskip

\noindent
{\em Question.\/} Do graph hypersurfaces generate the localized Grothendieck
ring as a module over $\Zbb[\Lbb,\Lbb^{-1}]$?\medskip 

This situation illustrates a sharp dichotomy in the behavior of graph hypersurfaces
in the Grothendieck ring under the contrasting operations of inverting $\bL$
and setting $\bL$ to zero. Graph hypersurfaces are `as general as possible' 
after localization (at $\bL$ and $S'$), while they are extremely special with 
respect to taking the quotient modulo the ideal $(\Lbb)$: as we will show, their 
span agrees with the span of a point modulo~$(\Lbb)$.
\smallskip

We quickly recall some basic notation and terminology. For a connected finite graph
$G$ with $n$ edges the graph polynomial $\psi_G(t_1,\ldots,t_n)$ is defined as
$$ \psi_G(t_1,\ldots,t_n)= \sum_{T\subseteq G} \prod_{e\notin T} t_e, $$
where $T$ runs over the spanning trees of $G$ and $t_e$ is the variable associated
to an edge $e$. 
In general, we define the graph polynomial for a (finite) graph~$G$ to be the product 
of the polynomials for the connected components of $G$.
We denote by $X_G$ the projective hypersurface defined by
the homogeneous polynomial $\psi_G$ in $\P^{n-1}$, by $\widehat X_G \subset \A^n$
the affine hypersurface, and by $Y_G\subseteq \A^n$ the affine hypersurface 
complement $Y_G = \A^n \smallsetminus \widehat X_G$. 
As the main results of \cite{BeBro} are expressed in terms of $Y_G$, we 
choose to focus on $Y_G$ in this paper. We note that graph hypersurfaces
are usually singular; also, it is easy to see that the irreducible components
of $X_G$ are rational.\smallskip

Our main result can be stated as follows. Larsen and Lunts associate with
each variety $V$ (possibly singular, possibly non-compact) an element in
the monoid ring $\Zbb[SB]$ generated by {\em stable birational equivalence\/}
classes of varieties. This assignment is compatible with the relations 
defining $K(\Var)$, and associates with $V$ its own stable birational equivalence
class {\em if $V$ is smooth and projective.\/} Smooth projective rational varieties 
have class $1$ in~$\Zbb[SB]$, but note that the element in $\Zbb[SB]$ 
determined by the image of the class $[V]\in K(\Var)$ of a
quasi-projective or singular rational variety need not
be in the `constant' part $\Zbb\subseteq \Zbb[SB]$ in general 
(cf.~Example~\ref{Eexample}). Thus, although irreducible graph hypersurfaces
are rational, this fact alone does not give information on their image in $\Zbb[SB]$.
What we show is precisely that the Larsen-Lunts image of graph hypersurfaces 
{\em do\/} lie in the constant part of $\Zbb[SB]$.

\begin{thm}\label{SBmain}
Affine graph hypersurface complements span $\Zbb\subseteq \Zbb[SB]$.
\end{thm}

Morally, Theorem~\ref{SBmain} shows that graph hypersurfaces and 
their complements are rational in a very strong sense (which we will make
precise in~\S\ref{SBsec}): for example, the image in $\Zbb[SB]$ 
of the class of an irreducible graph hypersurface does equal the class 
of a point (Corollary~\ref{Lrat}).
This is the reason why they do not span the
unlocalized Grothendieck ring of varieties.

In more classical terms, Theorem~\ref{SBmain} provides some
information on the mixed Hodge structure of graph hypersurfaces, as
it shows that the Deligne-Hodge polynomial of a graph hypersurface is
necessarily of the form $c+uv P(u,v)$, with $c\in \Zbb$. This shows that 
elliptic curves are not in the span of graph hypersurfaces in the Grothendieck 
ring.

Theorem~\ref{SBmain} is proven by using the realization of $\Zbb[SB]$ 
as the quotient $K(\Var)/(\Lbb)$ (also recalled in \S\ref{SBsec}).
An explicit computation (Theorem~\ref{modLthm}) based on a 
deletion-contraction formula for
the Grothendieck class of a graph hypersurface shows that for every graph $G$,
the class of $Y_G$ modulo $(\Lbb)$ is $0$ or~$\pm 1$. Theorem~\ref{SBmain}
and the consequences mentioned above follow immediately. 

\section{Stable birational equivalence and the Grothendieck ring}\label{SBsec}

In the following, we denote by $K(\Var)$ the Grothendieck ring of
varieties. This is generated by the isomorphism classes of irreducible
quasi-projective varieties with the inclusion--exclusion relations 
$[X]=[X\smallsetminus Y]+[Y]$, for closed embeddings $Y\subseteq X$, 
and with the product $[X]\cdot [Y]= [X\times Y]$.
The Grothendieck ring $K(\Var)$ depends on the field 
of definition of the varieties.
This will be understood to be $\Qbb$ in the following.

A result of \cite{LaLu} relates the Grothendieck ring to the ring of 
stable birational equivalence classes of varieties. We comment here
briefly on some aspects of this result that will be useful in the discussion
of the case of graph hypersurfaces.

Two (irreducible, complex) varieties $X$, $Y$ are {\em stably birational} if
$X\times \Pbb^k$ is birational to $Y\times \Pbb^\ell$ for some $k,\ell\ge 0$. 
If $X$ is stably birational to $Y$, and $X'$ is stably birational
to $Y'$, then $X\times X'$ is 
stably birational to $Y\times Y'$. 
Thus, the set of classes of stable birational equivalence of varieties is a 
multiplicative monoid $SB$, with unit equal to the class of a point.

\begin{thm}\label{LaLuThm} {\rm (\cite{LaLu}, Proposition~2.7.)}
Let $(\bL)$ be the ideal in $K(\Var)$ generated by the Lefschetz motive
$\bL=[\A^1]$. The ring 
$K(\Var)/(\Lbb)$ is isomorphic to the monoid ring $\Zbb[SB]$. 
\end{thm}

This result is obtained by defining a homomorphism $K(\Var) \to \Zbb[SB]$,
sending $[V]$ to the stable birational equivalence class $[V]_{SB}$
of $V$ for every
irreducible {\em smooth, projective\/} variety $V$.
The main technical step is to show that this homomorphism is
well-defined; this may be proven by using Bittner's alternative description
(\cite{Bitt}) of the Grothendieck ring of varieties $K(\Var)$ with generators that 
are smooth projective varieties and relations 
$$ [X] - [Y] = [B\ell_Y(X)]-[E], $$
for a smooth closed subvariety $Y\subseteq X$, with $B\ell_Y(X)$ the blowup of 
$X$ along~$Y$ and $E$ the exceptional divisor. This relation replaces the 
usual inclusion-exclusion relation $[X]=[X\smallsetminus Y]+[Y]$, which 
requires the noncompact $X\smallsetminus Y$.
Bittner's characterization depends on the {\em weak factorization theorem\/}
of \cite{AKMW}, which shows that any proper birational map between smooth 
irreducible varieties over a field of characteristic zero can be factored into a 
sequence of blow-ups and blow-downs with smooth centers.

Theorem~\ref{LaLuThm} is stated over $\Cbb$ in~\cite{LaLu}, but holds
over~$\Qbb$ as well since so does the weak factorization theorem
(Remark~2 after Theorem~0.3.1 of~\cite{AKMW}); this is also observed
explicitly in~\cite{Kollar}, p.~28.

\begin{rem}\label{LLmis} {\rm 
Stable birational equivalence makes sense for every variety~$V$, so every
variety (whether or not smooth and projective) has a class $[V]_{SB}$
in $\Zbb[SB]$.
It is important to keep in mind that in general this class agrees with the 
image of $[V]\in K(\Var)$ via the Larsen-Lunts homomorphism
{\em only if\/} $V$ is smooth and projective.
In other cases the image of $[V]$ in $\Zbb[SB]$ may be determined by
expressing $[V]$ as a combination of classes of smooth projective 
varieties, and then reproducing that combination in $\Zbb[SB]$.

For example, the image of $\Lbb=[\Abb^1]$ is $0$ in $\Zbb[SB]$
because $[\Abb^1]=[\Pbb^1]-[\Pbb^0]$ in $K(\Var)$, and $\Pbb^0$,
$\Pbb^1$ are trivially stably birationally equivalent to each other.
Likewise, the image of (the class of) an irreducible nodal plane 
cubic~$C$ in $\Zbb[SB]$ is~$0\ne 1$ even though  
$[C]_{SB}=[\Pbb^0]_{SB}$,
since $[C]=[\Pbb^1]-[\Pbb^0]$ in $K(\Var)$.
}\end{rem}

To make this point more explicit, we introduce a notion of 
`$\bL$-equivalence'. 

\begin{defn}\label{LLbirateq}
Two irreducible quasi-projective varieties $X$, $Y$ are 
$\bL$-equivalent if their classes in $\Zbb[SB]$ via the Larsen-Lunts
isomorphism coincide, that is,  if $[X]\equiv [Y] \mod (\Lbb)$ in $K(\Var)$.
A variety is $\bL$-rational if it is $\bL$-equivalent to $\Pbb^k$, for some 
$k\ge 0$.
\end{defn}

If $X$ and $Y$ are irreducible smooth projective and stably birational,
then they are also $\bL$-equivalent; however, this is not necessarily 
the case if $X$, $Y$ are not smooth and/or not complete. For example,
as observed above, 
an irreducible nodal cubic in $\Pbb^2$ is complete and birational to~$\Pbb^1$ 
but not $\bL$-rational.

In fact, the following example shows that rational (singular, projective)
varieties may be very far from being $\bL$-rational.

\begin{ex}\label{Eexample}{\rm
There exists a complete rational surface $X$ whose Larsen-Lunts
image in $\Zbb[SB]$ is $2-[C]_{SB}$, where $C$ is an elliptic curve.

Indeed, by Theorem~3.3 of~\cite{GreVi}, there exist projective rational
surfaces $X$ with one isolated singular point $p$ such that the exceptional
divisor in the minimal resolution $\widetilde X$ of $X$ is an elliptic curve $C$.
Using Bittner's relations, $[X]=[\widetilde X]-[C]+[p]$ in $K(\Var)$, and since 
all varieties on the right-hand side are smooth and projective, and 
$\widetilde X$ is rational, then the image of $[X]$ in $\Zbb[SB]$ equals 
$[\widetilde X]_{SB}-[C]_{SB}+[p]_{SB}=2-[C]_{SB}$.
}\end{ex}

These caveats apply to graph hypersurfaces. As recalled in
the introduction, irreducible (projective) graph hypersurfaces $X_G$
are easily seen to be rational, and are
complete, but are in general singular. Affine graph hypersurface 
complements $Y_G$ are trivially
rational, but non-complete. As observed above, these naive considerations
do not suffice to determine the Larsen-Lunts images of these varieties
in~$\Zbb[SB]$.

\section{Graph hypersurfaces and stable birational equivalence}

In \cite{AluMa4}, we proved a deletion-contraction formula for
the classes in the Grothendieck ring of the graph hypersurfaces. 
We recall here the result, for completeness, since the conclusion
we derive on the stable birational equivalence classes will be
a direct consequence of this formula.
We also note that the result is also implicit in the literature preceding
\cite{AluMa4}, e.g.~\cite{Ste},~\cite{BEK}.

\begin{thm}\label{delconthm} {\rm (\cite{AluMa4}, Theorem~3.8)}
Let $G$ be a graph with $n$ edges, and let $e$  be an edge of $G$. 
Denote by $G\smallsetminus e$ the graph obtained by removing $e$, 
and by $G/e$ the graph obtained by contracting $e$.
Let $Y_G$ denote the affine graph hypersurface complement in $\A^n$. 
\begin{itemize}
\item If $e$ is a bridge in $G$, then
$[Y_G]=\Lbb\cdot [Y_{G\smallsetminus e}]$;
\item If $e$ is a looping edge in $G$, then
$[Y_G]=(\Lbb-1)\cdot [Y_{G\smallsetminus e}]$;
\item If $e$ is a neither a bridge nor a looping edge in $G$, then
\[
[Y_G]=\Lbb\cdot [\Abb^{n-1}\smallsetminus Z]- [Y_{G\smallsetminus e}]\quad,
\]
where $Z$ is the intersection of the affine graph hypersurfaces for
$G\smallsetminus e$, $G/e$ in~$\Abb^{n-1}$.
\end{itemize}
\end{thm}

\proof We give a quick proof for completeness; for more details, see \cite{AluMa4}.
Let $t_i$ be the variable corresponding to the $i$-th edge $e_i$, and assume that
$e=e_n$ is not a bridge or a looping edge. In terms of graph polynomials:
\[
P_G(t_1,\dots,t_n) = t_n P_{G\smallsetminus e}(t_1,\dots,t_{n-1})
+P_{G/ e}(t_1,\dots,t_{n-1})\quad.
\]
The complement $Y_G$ is the set of $n$-tuples $(t_1,\dots,t_n)$ such
that $P_G(t_1,\dots,t_n)$ does not vanish; thus, such that
\[
t_n P_{G\smallsetminus e}(t_1,\dots,t_{n-1})
\ne -P_{G/ e}(t_1,\dots,t_{n-1})\quad.
\]
Over $Y_{G\smallsetminus e}$ (that is, if $P_{G\smallsetminus e}\ne 0$),
this condition holds if and only if $t_n\ne -P_{G/e}/P_{G\smallsetminus e}$,
that is, for $t_n\in \Abb^1\smallsetminus \Abb^0$.
Over the rest of the complement of~$Z$ (that is, if $P_{G\smallsetminus e}
\ne 0$ but $P_{G/e}=0$), the condition is satisfied for all $t_n$, hence
for $t_n\in \Abb^1$.
Over $Z$ (that is, if $P_{G\smallsetminus e}=P_{G/e}=0$), the condition
is not satisfied for any choice of $t_n$. Therefore,
\[
[Y_G]=(\Lbb-1)\cdot [Y_{G\smallsetminus e}]+\Lbb\cdot 
[(\Abb^{n-1}\smallsetminus Z)\smallsetminus Y_{G\smallsetminus e}]
\quad,
\]
and this is equivalent to third equality stated above. The other cases are
analogous.
\endproof

The deletion-contraction formula yields the following computation
of the Larsen-Lunts image of $Y_G$ in $\Zbb[SB]\cong K(\Var)/(\Lbb)$.

\begin{thm}\label{modLthm}
\[
[Y_G]\equiv \begin{cases}
\quad 0\quad  \mod (\Lbb) \quad  \text{if $G$ has edges that are not looping edges} \\
(-1)^n  \mod (\Lbb) \quad  \text{if $G$ has $n\ge 0$ looping edges, and no
other edge.}
\end{cases}
\]
\end{thm}

\proof Reading the result of Theorem \ref{delconthm} above modulo $(\Lbb)$ gives:
\begin{itemize}
\item If $e$ is a bridge in $G$, then
$[Y_G]=0 \mod (\Lbb)$;
\item If $e$ is not a bridge $G$, then
$[Y_G]=- [Y_{G\smallsetminus e}] \mod (\Lbb)$.
\end{itemize}
If now $G$ has any edge that is not a looping edge,
then removing all but one such edge leaves a graph with a bridge, and
hence $[Y_G]\equiv 0\mod (\Lbb)$ in this case. If all $n$ edges of $G$ are looping
edges, then repeated application of the second formula shows that
$[Y_G]\equiv(-1)^n \cdot [Y_{\underline G}]\mod (\Lbb)$, where $\underline G$ is 
the graph obtained by removing all edges from $G$. Clearly $Y_{\underline G}
=\Abb^0$, so $[Y_G]\equiv(-1)^n \mod (\Lbb)$ in this case, as stated.
\endproof

In terms of projective graph hypersurfaces:

\begin{cor}\label{Lrat}
Let $G$ be a graph that is not a forest, and with at least one non-looping edge. 
Then the projective graph hypersurface $X_G$ is $\Lbb$-rational.
\end{cor}

\begin{proof}
Note that $G$ must have at least~$2$ edges.
Since $G$ is not a forest, $\widehat X_G$ is not empty; and by 
Theorem~\ref{modLthm}, since $G$ has non-looping edges, then the class 
$[Y_G]$ is in the ideal $(\Lbb)$. The affine complement 
$Y_G=\A^n\smallsetminus \widehat X_G$ fibers over the projective 
complement 
$\Pbb^{n-1}\smallsetminus X_G$, with fibers $\Abb^1\smallsetminus \Abb^0$. 
Therefore
\[
(\Lbb-1)\cdot [\Pbb^{n-1}\smallsetminus X_G] \in (\Lbb)\quad,
\]
and hence $[\Pbb^{n-1}\smallsetminus X_G]\in (\Lbb)$; thus
\[
[X_G]\equiv [\Pbb^{n-1}] \mod (\Lbb)\quad,
\]
showing that $X_G$ is $\Lbb$-equivalent to $\Pbb^{n-1}$.
\end{proof}

Again we remark that if $X_G$ is irreducible, then it is easily seen to be rational,
but examples such as 
Example~\ref{Eexample} show that this does not suffice in itself
to draw the conclusion stated in Corollary~\ref{Lrat}.\smallskip

In view of the result of \cite{LaLu} recalled in Theorem~\ref{LaLuThm} above, 
Theorem~\ref{modLthm} has the following immediate consequence:

\begin{cor}\label{LLcor}
Let $\cR$ be the subring of $K(\Var)$ spanned by the classes~$[Y_G]$.
Then the image of $\cR$ in $\Zbb[SB]$ via the Larsen-Lunts homomorphism
is the subring $\Zbb$ generated by the stable birational equivalence class 
of a point.
\end{cor}

\begin{proof}
The image of $\cR$ in $\Zbb[SB]$ is the quotient $\cR/(\Lbb)$. 
By Theorem~\ref{modLthm}, $\cR/(\Lbb)\cong \Zbb$.
\end{proof}

\begin{rem}{\rm
With our conventions, the product of two classes $[Y_{G_1}]$, $[Y_{G_2}]$
is itself the class $[Y_{G_1\amalg G_2}]$ of the affine complement of
a graph hypersurface (the class of the affine hypersurface complement is
a `motivic Feynman rule', see Proposition~2.5~in~\cite{AluMa2}). Further,
the Lefschetz motive $\Lbb$ equals $[Y_G]$ for the graph $G$ consisting
of a single edge joining two distinct vertices. Thus, the ring $\cR$ generated
by the classes $[Y_G]$ agrees with the $\Zbb[\Lbb]$-module generated
by the classes $[Y_G]$. Therefore, the following immediate consequence
of Corollary~\ref{LLcor} formalizes the first `non-spanning' result mentioned
in the Introduction.
}\end{rem}

\begin{thm}\label{notspanK}
The classes $[Y_G]$ of the affine graph hypersurface complements do not
span the Grothendieck ring $K(\Var)$ over $\Zbb[\Lbb]$.
\end{thm}

\proof In view of Corollary~\ref{LLcor}, it suffices to notice that $\Zbb[SB]\neq \Zbb$.
This is clear since there are polynomial
invariants of smooth projective varieties that are invariant under stable
birational equivalence, and are not always constant. One example is
given in \cite{LaLu}, Definition~3.4 in terms of Hodge polynomials.
\endproof

As mentioned in the Introduction, this observation can be sharpened, in a way 
that relates well to the result of \cite{BeBro}.

\begin{cor}\label{BelBrocor}
Let $S'$ be the saturated multiplicative subset of $\Zbb[\Lbb]$ generated by
the elements $\Lbb^n-1$, for $n>0$. Then the classes $[Y_G]$ do not
generate ${S'}^{-1} K(\Var)$ over ${S'}^{-1}\Zbb[\Lbb]$.
\end{cor}

\begin{proof}
Localization commutes with taking quotients: the quotient 
$$({S'}^{-1} K(\Var))/(\Lbb)$$
equals the localization ${S'}^{-1}(K(\Var)/(\Lbb))$. 
Since all elements of $S'$ are invertible 
modulo $(\Lbb)$, this latter equals $K(\Var)/(\Lbb)$; and the action of $\Lbb$
on this module is $0$. Since the classes $[Y_G]$ do not span this quotient,
they cannot span ${S'}^{-1} K(\Var)$ over ${S'}^{-1} \Zbb[\Lbb]$.
\end{proof}

Corollary~\ref{BelBrocor} should be compared with Theorem~0.6 in
\cite{BeBro}, which states that the classes $[Y_G]$ {\em do\/} generate
the localization $S^{-1} K(\Var_\Zbb)$ over $S^{-1}\Zbb[\Lbb]$, where $S$ 
is the multiplicative system generated by $\Lbb^n-\Lbb$ for $n>1$. 
{\em A fortiori,\/} the classes~$[Y_G]$ generate $S^{-1} K(\Var_\Qbb)$. 
Localizing at $S$ amounts to localizing at $S'$ and at $\Lbb$; 
Corollary~\ref{BelBrocor} shows that the localization at~$\Lbb$ is crucial to 
the mentioned result in~\cite{BeBro}. This suggests that the classes $[Y_G]$
may possibly span the Grothendieck ring over the simpler 
localization $\Zbb[\Lbb^{-1},\Lbb]$.\smallskip

We end by observing that the results stated above have a straightforward
Hodge-theoretic formulation.
Every smooth complex projective variety $X$ carries Hodge numbers 
$h^{p,q}(X)$. The {\em Hodge polynomial\/} of $X$ is the polynomial 
$\sum_{p,q} (-1)^{p+q} h^{p,q}(X)u^pv^q$. Now, the Hodge
polynomial determines a ring homomorphism $K(\Var) \to \Zbb[u,v]$,
see e.g.~\S2.11 in \cite{Veys}; this homomorphism maps $\Lbb$
to $uv$. Indeed, the Hodge polynomial may be consistently 
defined for all varieties and satisfies the relations in $K(\Var)$, 
as observed in~\cite{DK}. This (`Deligne-Hodge') polynomial of
an arbitrary complex variety $X$ records information about the mixed 
Hodge structure on the cohomology $H^k_c(X,\Qbb)$ of $X$ with
compact support: it may be defined as $\sum_{p,q} e^{p,q}(X)
u^p v^q$, where
\[
e^{p,q}(X)=\sum_k (-1)^k h^{p,q} (H_c^k(X,\Qbb))\quad.
\]

\begin{cor}\label{hodgecor}
Let $X$ be any complex projective variety whose class $[X]$ is in the
subring $\cR$ of $K(\Var_\C)$ generated by the classes $[Y_G]$, as $G$ ranges
over all graphs. Then the Deligne-Hodge polynomial of $X$ is of the form
\begin{equation}\label{Hodgepoly}
c+uvP(u,v)\quad,
\end{equation}
with $c\in \Zbb$.
\end{cor}

\begin{proof}
The Hodge polynomial induces a homomorphism
$$ K(\Var)/(\Lbb) \to \Zbb[u,v]/(uv). $$ The image of $\cR/(\Lbb)\cong \Zbb$ in
$\Zbb[u,v]/(uv)$ is $\Zbb$, and the statement follows.
\end{proof}

For example, $h^{1,0}(X)\ne 0$ for e.g., a smooth elliptic curve. Therefore,
Corollary~\ref{hodgecor} shows that classes of elliptic curves are not in 
the span of the classes $[Y_G]$.

\end{document}